\documentclass[reqno, a4paper, 10pt, draft]{amsart}
\usepackage{amsmath,amsthm,amssymb,amscd,amstext,amsopn,amsxtra,amsfonts, eucal, bbm}

\usepackage[all]{xy}
\SelectTips{eu}{}
\entrymodifiers={+!!<0pt,\fontdimen22\textfont2>}

\usepackage[pdftex]{hyperref}
\hypersetup{
linktocpage
}

\swapnumbers

\newtheorem{thm}{Theorem}[subsection]
\newtheorem{lemma}[thm]{Lemma}
\newtheorem{prop}[thm]{Proposition}
\newtheorem{cor}[thm]{Corollary}

\theoremstyle{remark}

\theoremstyle{definition}

\newtheorem{remark}[thm]{Remark}
\newtheorem{remarks}[thm]{Remarks}

\numberwithin{equation}{section}

\def\e{\mathbbm{1}}

\def\cA{\mathcal{A}}
\def\cB{\mathcal{B}}
\def\cC{\mathcal{C}}

\def\cO{\mathcal{O}}

\def\cR{\mathcal{R}}

\def\CC{\mathbf{C}}

\def\QQ{\mathbf{Q}}

\def\ZZ{\mathbf{Z}}

\def\fh{\mathfrak{h}}

\def\fn{\mathfrak{n}}

\def\fgl{\mathfrak{gl}}  
\def\fsl{\mathfrak{sl}}

\def\Comp{\mathrm{Comp}}

\def\Db{\mathrm{D^b}}

\def\End{\mathrm{End}}

\def\Hob{\mathrm{Ho^b}}
\def\Hom{\mathrm{Hom}}

\def\id{\mathrm{id}}

\def\pr{\mathrm{pr}}

\def\tr{\mathrm{tr}}

\newcommand{\mapright}[1]{\stackrel{#1}{\longrightarrow}}

\makeatletter
\renewcommand{\@makefnmark}{\mbox{\textsuperscript{}}}
\makeatother

\title{Derived equivalences and $\fsl_2$-categorifications for $U_q(\fgl_n)$}
\author{R. Virk}
\address{Department of Mathematics\\
University of Wisconsin\\
Madison, WI 53706}
\email{virk@math.wisc.edu}

\begin{document}

\maketitle
%\begin{abstract}
%We give a construction of $\fsl_2$-categorifications for $U_q(\fgl_n)$, for generic $q$ and for $q$ a root of unity.
%\end{abstract}
\setcounter{tocdepth}{1}
\tableofcontents

\renewcommand{\thepart}{\textbf{\Roman{part}}}
\section{Introduction}
The notion of $\fsl_2$-categorifications was introduced by J.\ Chuang and R.\ Rouquier in \cite{ChRo}. In \emph{op. cit.} it is explained how an $\fsl_2$-categorification on an abelian category $\cA$ leads to an auto-equivalence of the derived category $\Db(\cA)$. This is in turn elegantly exploited to give a proof of Brou\'e's abelian defect group conjecture for blocks of symmetric groups \cite[Thm.\ 7.6]{ChRo}. 

The $q$-Schur algebra is the centralizer algebra of a certain module for the finite Hecke algebra of type A (see \cite[\S7.6]{ChRo} for the precise definition). In \cite{ChRo}, J.\ Chuang and R.\ Rouquier also explain how $\fsl_2$-categorifications may be obtained for the $q$-Schur algebras. This is done by constructing $\fsl_2$-categorifications for the finite Hecke algebra and then transferring these to the $q$-Schur algebra using a `double centralizer theorem' (see \cite[\S7.6]{ChRo}).

The purpose of this note is to give a direct construction of $\fsl_2$-categorifications for the quantum group $U_q(\fgl_n)$ (Prop.\ \ref{categorificationgenericq} and Prop.\ \ref{categorificationrootsof1}). One motivation for doing so is Thm.\ \ref{dequivgenericq} which is in the same spirit as \cite[Thm.\ 7.24]{ChRo}. A related example fitting into this setup can be found in \cite{BS} (see Remark 5.7 therein). 

\section{Monoidal categories}
The purpose of this section is to precisely define the type of categories we will be working with. This avoids all possible confusion as to `how strict' our monoidal structures are. The exposition mainly follows \cite{Tu} (also see \cite{CP}).

\subsection{}\label{ssdefstrictmonoidal}
A \emph{strict monoidal category} is a $3$-tuple $\cC=(\cC,\otimes, \e)$ consisting of a category $\cC$, a bifunctor $\otimes\colon \cC\times \cC \to \cC$, a distinguished object $\e\in\cC$ called the \emph{unit}, satisfying the following:

\begin{enumerate}
	\item For all $V\in\cC$, $\e\otimes V=V$ and $V\otimes\e = V$.
\item Let $X_1$ and $X_2$ are two expressions obtained from $V_1\otimes V_2\otimes \cdots \otimes V_m$ by inserting $\e$'s and gramatically correct parentheses: an example of such an expression is
		$(V_1\otimes \e)\otimes ((V_2\otimes V_3)\otimes\cdots \otimes V_m)$.
Then $X_1=X_2$.
\end{enumerate}

\begin{remark}Almost all the examples of monoidal categories that arise `in nature' are non-strict (for example, vector spaces). That is, all the equalities in the defining axioms need to be replaced by functorial isomorphisms. Fortunately, it is known that every monoidal category is equivalent to a strict one \cite[Ch.\ XI \S3, Thm.\ 1]{MacL}. This result will constantly be invoked to omit parentheses and the associativity and unit isomorphisms in our formulas even when dealing with non-strict monoidal categories. 
\end{remark}

\subsection{}
		A \emph{braiding} or \emph{$R$-matrix} in a strict monoidal category $\cC$ is a collection of isomorphisms
		\[ \cR_{VW}\colon V\otimes W {\mapright \sim} W\otimes A, \]
		for all $V,W\in\cC$, satisfying:
		
		\begin{enumerate}
			\item For every $f\colon X\to X'$ and $g\colon Y\to Y'$ in $\cC$, the diagram
	\begin{equation}\label{RfunctorialX}
					\xymatrixrowsep{1.5pc}\xymatrixcolsep{2pc}\xymatrix{ X\otimes Y\ar[d]_{\cR_{XY}}\ar[r]^{f\otimes g}& X'\otimes Y'\ar[d]^{\cR_{X'Y'}} \\
					Y\otimes X\ar[r]^{g\otimes f}&Y'\otimes X'}
				\end{equation}
				commutes.
			\item For every $X\in\cC$, $\cR_{X\e}=\id=\cR_{\e X}$. In particular, $\cR_{\e\e}=\id$.
			\item The following two diagrams commute for all $X,Y,Z\in\cC$:
		\begin{equation}\label{Rtriangle1X}\xymatrixcolsep{1pc}\xymatrixrowsep{1.5pc}\xymatrix{
&X\otimes Y\otimes Z\ar@/^0pc/[ld]_{\id\otimes \cR_{YZ}}\ar@/^0pc/[rd]^{\cR_{(X\otimes Y)Z}} \\
X\otimes Z\otimes Y \ar[rr]_{\cR_{XZ}\otimes\id}&&Z\otimes X\otimes Y
}\end{equation}
\begin{equation}\label{Rtriangle2X}\xymatrixcolsep{1pc}\xymatrixrowsep{1.5pc}\xymatrix{
&X\otimes Y\otimes Z\ar@/0pc/[ld]_{\cR_{XY}\otimes\id}\ar@/0pc/[rd]^{\cR_{X(Y\otimes Z)}}\\
Y\otimes X\otimes Z \ar[rr]_{\id\otimes\cR_{XZ}}&&Y\otimes Z\otimes X 
}\end{equation}
\end{enumerate}
Combining \eqref{Rtriangle1X} and \eqref{Rtriangle2X} we obtain the `hexagon diagram', i.e.
\begin{equation}\label{hexagonX}\xymatrixrowsep{1.5pc}\xymatrixcolsep{1pc}\xymatrix{
&Y\otimes X\otimes Z\ar@/0pc/[rd]^{\id\otimes\cR_{XZ}}& \\
X\otimes Y\otimes Z\ar@/0pc/[ur]^{\cR_{XY}\otimes \id}\ar[d]_{\id\otimes\cR_{YZ}}\ar[rr]^{\cR_{X(Y\otimes Z)}}&&Y\otimes Z\otimes X\ar[d]^{\cR_{YZ}\otimes \id} \\
X\otimes Z\otimes Y\ar@/0pc/[rd]_{\cR_{XZ}\otimes \id}\ar[rr]^{\cR_{X(Z\otimes Y)}}&&Z\otimes Y\otimes X \\
&Z\otimes X\otimes Y\ar@/0pc/[ur]_{\id\otimes\cR_{XY}}&
}\end{equation}
commutes for all $X,Y,Z\in\cC$.

\subsection{}\label{ssdeffduals}
Let $\cC$ be a strict monoidal category and let $V\in\cC$. A \emph{left dual} to $V$ is an object $V^*\in \cC$ along with two morphisms
\begin{equation}
	\varepsilon_V\colon V^*\otimes V\to \e \qquad\mbox{and}\qquad
	\eta_V\colon \e\to V\otimes V^*,
\end{equation}
such that the compositions
\begin{equation}\label{dualityunit}\xymatrixcolsep{3.5pc}\xymatrix{V=\e\otimes V \ar[r]^-{\eta_V\otimes\id}& V\otimes V^*\otimes V \ar[r]^-{\id\otimes\varepsilon_V}& V\otimes \e=V, }\end{equation}
\begin{equation}\label{dualitycounit}\xymatrixcolsep{3.5pc}\xymatrix{ V^*=V^*\otimes\e\ar[r]^-{\id\otimes\eta_V}& V^*\otimes V\otimes V^* \ar[r]^-{\varepsilon_V\otimes\id}& \e\otimes V^*=V^*}\end{equation}
are equal to the identity.
Similarly, the \emph{right dual} of an object $V$ is an object $V^{\circledast}\in \cC$ along with morphisms
\begin{equation} \varepsilon_V'\colon V\otimes V^{\circledast} \to \e \qquad\mbox{and}\qquad \eta_V'\colon\e\to V^{\circledast}\otimes V, \end{equation}
satisfying the obvious analogues of the identities in the previous definition.

\begin{remark}
In \cite{Tu} only $V^*$ is considered and is simply called the dual of $V$. Our choice of `left/right' is justified by the fact that the functor $V^*\otimes -$ is left adjoint to $V\otimes -$. Similarly $V^{\circledast}\otimes -$ is right adjoint to $V\otimes -$.
\end{remark}

\subsection{}
A \emph{ribbon} category $\cC$ is a braided category equipped with an automorphism $\theta=\{\theta_X\colon X{\mapright \sim}X\,|\, X\in\cC\}$ of the identity functor, satisfying:
\begin{equation}\label{casimir1}
	\theta_{X\otimes Y}\circ(\theta_X\otimes\theta_Y)^{-1}=\cR_{YX}\circ \cR_{XY},
	\end{equation}
for all $X,Y\in\cC$.

\section{Affine braid group action in ribbon categories}\label{s:braidaction}
\subsection{}Let $\cC$ be a ribbon category. Fix an object $V$ in $\cC$. Denote by $F_V$ the functor $V\otimes -\colon\cC\to\cC$.
Define an endomorphism $Y=\{Y_M\colon F_V(M)\to F_V(M)\,|\, M\in\cC\}$
of $F_V$ by
\[ Y_M(V\otimes M)= \cR_{MV}\circ\cR_{VM}(V\otimes M). \]
Define an endomorphism $\sigma=\{\sigma_M\colon F_V^2(M)\to F_V^2(M)\,|\, M\in\cC\}$ of $F_V^2$ by
\[ \sigma_M(V^{\otimes 2}\otimes M)=(\cR_{VV}\otimes\id)(V^{\otimes 2}\otimes M). \]
Let $k\in\ZZ_{>0}$ and define endomorphisms $Y_i$, $1\leq i\leq k+1$ and $\sigma_i$, $1\leq i \leq k$ of $F_V^{k+1}$ by
\[
	Y_i=\id^{\otimes k-i+1}\otimes Y_{V^{\otimes i-1}\otimes M}, \qquad
	\sigma_i=\id^{\otimes k-i}\otimes \sigma_{V^{\otimes i-1}\otimes M}.
\]
The following is well known (for instance, see \cite{LR}).
\begin{prop}\label{braidrels}Let $M\in\cC$. The following equalities hold in $\End(F_V^{k+1}(M))$:
\begin{align}
	\sigma_i\sigma_j&=\sigma_j\sigma_i, \qquad\mbox{if $|i-j|>1$}; \label{coxeter1}\\
	\sigma_i\sigma_{i+1}\sigma_i&=\sigma_{i+1}\sigma_i\sigma_{i+1}, \label{coxeter2} \\
	\sigma_iY_i\sigma_i&=Y_{i+1},  \label{affine1} \\
	\sigma_iY_j&=Y_j\sigma_i, \qquad \mbox{if $|i-j|>1$}; \label{affine2} \\
	\sigma_iY_iY_{i+1}&=Y_{i+1}Y_i\sigma_i,  \label{affine3} \\
	Y_iY_j&=Y_jY_i. \label{affine4}
\end{align}
\end{prop}
\begin{proof}
The equality \eqref{coxeter1} is clear from the definition of the $\sigma_i$s. It suffices to verify \eqref{coxeter2} for $i=1$. In this case it is simply the outer arrows of \eqref{hexagonX} with $X=Y=Z=V$. To prove \eqref{affine1} it is again sufficient to do so for $i=1$. That is, we need to show
	\[ (\cR_{VV}\otimes \id)\circ (\id\otimes \cR_{MV})\circ (\id\otimes \cR_{VM})\circ (\cR_{VV}\otimes \id) = \cR_{(V\otimes M)V}\circ\cR_{V(V\otimes M)}. \]
	Set $X=Y=V$ and $Z=M$ in \eqref{Rtriangle2X} to get $(\id\otimes\cR_{VM})\circ (\cR_{VV}\otimes\id)=\cR_{V(V\otimes M)}$. Set $X=Z=V$ and $Y=M$ in \eqref{Rtriangle1X} to get $(\cR_{VV}\otimes\id)\circ (\id\otimes \cR_{MV})=\cR_{(V\otimes M)V}$. This gives the desired equality.

	The equality \eqref{affine2} is clear if $i>j+1$. If $j>1+i$ then \eqref{affine2} follows by applying \eqref{RfunctorialX} (twice) with $X=X'=V^{\otimes k-j+1}$, $Y=Y'=V^{\otimes j-1}$, $f=\id^{\otimes k-i}$, $g=\sigma_{V^{\otimes i-1}\otimes M}$. Equation \eqref{affine3} is immediate from \eqref{affine1}.
	
	Finally, to show \eqref{affine4} we may assume that $i>j$. If $i>j+1$, then the equality is clear from \eqref{coxeter1} and \eqref{affine1}. If $i=j+1$, then we are reduced to proving the claim for $j=1$. That is, we want to show
	\begin{align*}&\cR_{(V\otimes M)V}\circ\cR_{V(V\otimes M)}\circ (\id\otimes \cR_{MV})\circ (\id\otimes \cR_{VM}) \\
	= &(\id\otimes \cR_{MV})\circ (\id\otimes \cR_{VM})\circ \cR_{(V\otimes M)V}\circ \cR_{V(V\otimes M)}. \end{align*}
	Set $X=Z=V$ and $Y=M$ in the middle rectangle of \eqref{hexagonX} to get
	\begin{align*}
		&\quad\cR_{(V\otimes M)V}\circ\cR_{V(V\otimes M)}\circ (\id\otimes \cR_{MV})\circ (\id\otimes \cR_{VM})\\
		&= \cR_{(V\otimes M)V}\circ (\cR_{MV}\otimes\id)\circ \cR_{V(M\otimes V)}\circ (\id\otimes \cR_{VM}), \\
		\intertext{put $X=Y=V$ and $Z=M$ in the middle rectangle of \eqref{hexagonX} to obtain that this}
		&=\cR_{(V\otimes M)V}\circ(\cR_{MV}\otimes\id)\circ(\cR_{VM}\otimes\id)\circ\cR_{V(V\otimes M)}, \\
		\intertext{set $X=M\otimes V$, $X'=V\otimes M$, $Y=Y'=V$, $f=\cR_{MV}$ and $g=\id$ in \eqref{RfunctorialX} to get that this}
		&=(\id\otimes\cR_{MV})\circ\cR_{(M\otimes V)V}\circ(\cR_{VM}\otimes \id)\circ \cR_{V(V\otimes M)}, \\
		\intertext{finally, set $X=V\otimes M$, $X'=M\otimes V$, $Y=Y'=V$, $f=\cR_{VM}$ and $g=\id$ in \eqref{RfunctorialX} to get that this}
		&=(\id\otimes\cR_{MV})\circ(\id\otimes \cR_{VM})\circ\cR_{(V\otimes M)V}\circ\cR_{V(V\otimes M)}. \qedhere
\end{align*}
\end{proof}

\section{$\fsl_2$-categorifications}
We now summarize some of the results of \cite{ChRo} that we will be needed later.
\subsection{}
Let $q_0\in \CC^{\times}$. Assume $q_0\neq 1$. The \emph{affine Hecke algebra} $H_{k+1}=H_{k+1}(q_0)$ is the $\CC$-algebra with generators
$T_1^{\pm 1},\ldots, T_k^{\pm 1}, X_1^{\pm 1},\ldots, X_{k+1}^{\pm 1}$ subject to the relations
\begin{align*}
(T_i+1)(T_i-q_0) &=0,  \\
T_iT_i^{-1}&=1=T_i^{-1}T_i, \\
X_iX_i^{-1}&=1=X_i^{-1}X_i, \\
	T_iT_j&=T_jT_i, \qquad\mbox{if $|i-j|>1$}; \\
	T_iT_{i+1}T_i&=T_{i+1}T_iT_{i+1},  \\
	T_iX_iT_i&=q_0X_{i+1},  \\
	T_iX_j&=X_jT_i, \qquad \mbox{if $|i-j|>1$}; \\
	T_iX_iX_{i+1}&=X_{i+1}X_iT_i,  \\
	X_iX_j&=X_jX_i. \qedhere
\end{align*}

\begin{prop}\label{commutation1}Let $a\in\CC$, $N\in\ZZ_{>0}$. Then
\begin{align*}&T_1(X_2-a)^N - (X_1-a)^NT_1\\
= &(q_0-1)X_2((X_1-a)^{N-1} + (X_1-a)^{N-2}(X_2-a)+ \cdots + (X_2-a)^{N-1}). \end{align*}
\end{prop}

\begin{proof}If $N=1$, then the result is given by:
\[ T_1X_2 = q_0^{-1}T_1^2X_1T_1 = q_0^{-1}(q_0-(1-q_0)T_1)X_1T_1 = X_1T_1 + (q_0-1)X_2.\]
For arbitrary $N$, proceeding by induction, assume the statement for $N-1$. Then
\begin{align*}
T_1(X_2-a)^N &= (X_1-a)^{N-1}T_1(X_2-a) + (q_0-1)X_2((X_1-a)^{N-2}(X_2-a)\\
&\quad  + (X_1-a)^{N-3}(X_2-a)^2 + \cdots + (X_2-a)^{N-1}). \\
\intertext{Applying the computation for $N=1$, this is}
&=(X_1-a)^{N-1}((X_1-a)T_1 +(q_0-1)X_2) \\
&\quad+ (q_0-1)X_2((X_1-a)^{N-2}(X_2-a)\\
&\quad + (X_1-a)^{N-3}(X_2-a)^2 + \cdots + (X_2-a)^{N-1}) \\
&=(X_1-a)^NT_1 + (q_0-1)X_2((X_1-a)^{N-1}\\
&\quad  + (X_1-a)^{N-2}(X_2-a)+ \cdots + (X_2-a)^{N-1}).\qedhere
\end{align*}
\end{proof}

\subsection{}Let $S_{k+1}$ be the symmetric group on $k+1$ letters. Let $\ell\colon S_{k+1}\to\ZZ_{\geq 0}$ denote the length function. Let $s_i$ denote the simple transposition $(i,i+1)$. Given a reduced expression $w=s_{i_1}\cdots s_{i_r}$ for $w\in S_{k+1}$, put $T_w=T_{s_{i_1}}\cdots T_{s_{i_r}}$. The element $T_w$ is independent of the reduced expression.

\subsection{}Let $H_{k+1}^f$ denote the subalgebra of $H_{k+1}$ generated by the $T_i$. Define
\[\textbf{1}\colon H_{k+1}^f \to \CC, \quad T_i \mapsto q, \qquad
\mbox{and} \qquad \mathrm{sgn}\colon H_{k+1}^f \to \CC,\quad T_i \mapsto -1.\]
For $\tau\in\{\textbf{1}, \mathrm{sgn}\}$, put $c_{k+1}^{\tau} = \sum_{w\in S_n} q^{-\ell(w)}\tau(T_w) T_w$.

\subsection{}Let $\fsl_2$ be the Lie algebra of $2\times 2$ traceless matrices. It has a basis given by
\[ e=\begin{pmatrix} 0 & 1 \\ 0 & 0 \end{pmatrix},\qquad
f=\begin{pmatrix} 0 & 0 \\ 1 & 0 \end{pmatrix} \qquad \mbox{and} \qquad
h=ef-fe = \begin{pmatrix} 1 & 0 \\ 0 & -1\end{pmatrix}. \]
Set
\begin{align*} s&=\exp(-f)\exp(e)\exp(-f) =\begin{pmatrix} 0 & 1 \\ -1 & 0\end{pmatrix},\\
 s^{-1}&=\exp(f)\exp(-e)\exp(f)=\begin{pmatrix} 0 & -1 \\ 1 & 0\end{pmatrix}.\end{align*}
Let $U$ be a direct sum (possibly infinite) of finite dimensional $\fsl_2$-modules. Then the operator $s$ is well defined on $U$.
If $V$ and $W$ are $\fsl_2$ modules then so is $V\otimes W$ via
\[ x(v\otimes w) = xv\otimes w + v\otimes xw, \qquad x\in\fsl_2, v\in V, w\in W.\]

\subsection{}Let $\cA$ be an abelian category with the property that every object of $\cA$ is a successive extension of finitely many simple objects. Assume that the $\Hom$ groups in $\cA$ are in fact $\CC$-vector spaces and that composition is bilinear with respect to this structure (i.e., $\cA$ is enriched over $\CC$). Further, assume that the endomorphism ring of a simple object is $\CC$.
Write $K_0(\cA)$ for the Grothendieck group of $\cA$. By definition $K_0(\cA)$ is generated by symbols $[A]$, $A\in \cA$ and relations $[A]-[B]+[C] = 0$, for every exact sequence $0\to A \to B \to C \to 0$ in $\cA$.

\subsection{}Write $\Hob(\cA)$ for the homotopy category of bounded complexes in $\cA$ and $\Db(\cA)$ for the bounded derived category of $\cA$. The Grothendieck group $K_0(\Db(\cA))$ is generated by symbols $[A]$, $A\in \Db(\cA)$ and relations $[A]-[B]+[C]$, for every distinguished triangle $A\to B \to C \leadsto$. The map $K_0(\Db(\cA))\to K_0(\cA)$
given by
$[A]\mapsto \sum_i [H^i(A)]$,
where $H^{\bullet}(A)$ denotes the cohomology of the complex $A$, is an isomorphism. The inverse is given by the map induced by the embedding $\cA \hookrightarrow \Db(\cA)$. We identify $K_0(\Db(\cA))$ with $K_0(\cA)$ via this isomorphism.

\subsection{}An adjunction $(E, F)$ is the data of functors $E\colon\cA\to\cB$, $F\colon\cB\to\cA$ and morphisms
$\eta\colon \id_{\cA} \to FE$, $\varepsilon\colon EF \to \id_{\cB}$,
such that the compositions
\[\xymatrixcolsep{3.5pc}\xymatrix{F \ar[r]^-{\eta\e_F}& FEF \ar[r]^-{\e_F\varepsilon}& F }\qquad
\mbox{and}\qquad \xymatrixcolsep{3.5pc}\xymatrix{ E\ar[r]^-{\e_E\eta}& E F E \ar[r]^-{\varepsilon\e_E}& E}\]
are equal to the identity.
\begin{remark}If $\cA=\cB$, then the adjunction $(E,F)$ is precisely the data of a left dual (=left adjoint) $E$ to the functor $F$ in the monoidal category of endo-functors of $\cA$.
\end{remark}

\subsection{}A \emph{weak $\fsl_2$-categorification} is the data of an adjunction $(E,F)$ of exact endo-functors of $\cA$ such that
\begin{itemize}
\item the action of $e=[E]$ and $f=[F]$ on $V=\QQ\otimes K_0(\cA)$ gives a locally finite $\fsl_2$-representation;
\item the classes of simple objects are weight vectors;
\item $F$ is isomorphic to a left adjoint of $E$.
\end{itemize}
Denote by $\varepsilon\colon EF \to \id$ and $\eta\colon\id \to FE$ the (fixed) counit and unit of the adjunction $(E,F)$.

\begin{prop}\cite[Prop.\ 5.5]{ChRo}\label{blocks}Fix a weak $\fsl_2$-categorification $(E,F)$ on $\cA$.
Let $m\in \ZZ$ and let $V_m$ denote the $m$-weight space of $K_0(\cA)$ (viewed as a $\fsl_2$-module). Let $\cA_m$ denote the full subcategory of $\cA$ consisting of objects whose class is in $V_m$. Then, $\cA=\bigoplus_{m\in\ZZ} \cA_m$. In particular, the class of an indecomposable object of $\cA$ is a weight vector.
\end{prop}

\subsection{}An \emph{$\fsl_2$-categorification} is a weak $\fsl_2$-categorification with the extra data of $q_0\in \CC^{\times}$ and $a\in \CC$ with $a\neq 0$ if $q_0\neq 1$, and of $X\in \End(E)$ and $T\in \End(E^2)$ such that
\begin{itemize}
\item $(\e_ET)\circ(T\e_E)\circ(\e_ET) = (T\e_E)\circ (\e_ET)\circ(T\e_E)$;
\item $(T+\e_{E^2})\circ(T-q_0\e_{E^2})=0$;
\item $T\circ(\e_E X)\circ T = \begin{cases} q_0X\e_E & \mbox{if $q_0\neq 1$}, \\
X\e_E-T&\mbox{if $q_0=1$};\end{cases}$
\item $X-a$ is locally nilpotent;
\end{itemize}
where for a functor $G$, the symbol $\e_G$ denotes the identity transformation of $G$.

\subsection{}Given an $\fsl_2$-categorification, define a morphism $\gamma_n\colon H_{k} \to \End(E^k)$ by
\[ T_i \mapsto \e_{E^{k-i-1}}T\e_{E^{i-1}} \qquad \mbox{and}\qquad X_i\mapsto \e_{E^{n-i}} X \e_{E^{i-1}}. \]
Let $\tau\in\{\textbf{1},\mathrm{sgn}\}$. Put $E^{(\tau, k)} = E^kc_k^{\tau}$, the image of $c_k^{\tau}\colon E^n \to E^n$.

\subsection{}\label{ss:rickardcomplex}Let $m\in\ZZ$. Consider the complex of functors
\[ \Theta_m\colon \Comp(\cA_{-m})\to \Comp(\cA_m), \]
constructed as follows:
denote by $(\Theta_m)^{-r}$ the restriction of $E^{(\mathrm{sgn},m+r)}F^{(1,r)}$ to $\cA_{-m}$ for $r,-m+r\geq 0$. Put $(\Theta_{\lambda})^{-r}=0$ otherwise. The map
\[ \e_{E^{m+r-1}}\varepsilon \e_{F^{r-1}}\colon  E^{m+r-1}EFF^{r-1} \to E^{m+r-1}F^{r-1} \]
restricts to a map
$d^{-r}\colon E^{(\mathrm{sgn}, m+r)}F^{(1,r)} \to E^{(\mathrm{sgn},m+r-1)}F^{(1,r-1)}$.
Put
\[ \Theta_{\lambda} = \cdots \to (\Theta_{m})^{-i}\mapright{d^{-i}}(\Theta_m)^{-i+1}\to\cdots. \]
Then $\Theta_{\lambda}$ is a complex \cite[Lemma 6.1]{ChRo}.
Let
$\Theta=\bigoplus_{m\in\ZZ}\Theta_m$.
\begin{thm}\cite[Thm.\ 6.4]{ChRo}\label{reflectionequivalence}The complex $\Theta$ induces a self-equivalence of $\Hob(\cA)$ and of $\Db(\cA)$ and induces, by restriction, equivalences $\Hob(\cA_{-m})\mapright{\sim}\Hob(\cA_m)$ and $\Db(\cA_{-m})\mapright{\sim}\Db(\cA_m)$. Furthermore, the map induced by $\Theta$ on $K_0(\cA)$ coincides with the reflection $s$ on $K_0(\cA)$ (viewed as an $\fsl_2$-module).
\end{thm}

\section{Quantum $\fgl_n$}
\subsection{}Let $\fgl_n$ denote the Lie algebra of $n\times n$ matrices. Let $\fh$ denote the Lie subalgebra of $\fgl_n$ consisting of diagonal matrices and let $h_i$ be the matrix with $1$ in the $i^{\mbox{th}}$ diagonal entry and $0$ elsewhere. Let $\fh^*=\Hom_{\CC}(\fh,\CC)$. Define $\varepsilon_i\in\fh^*$ by $\langle \varepsilon_i, h_j\rangle=\delta_{ij}$. The \emph{trace form} $(\cdot|\cdot)$ on $\fgl_n$ is given by $(x|y)=\tr(xy)$, where $\tr$ is the ordinary matrix trace. This form is symmetric, ad-invariant and non-degenerate.
\subsection{}
Let $\fn$ be the Lie subalgebra of $\fgl_n$ consisting of strictly upper triangular matrices. By definition, the set of positive roots $R^+$ is the set of the eigenvalues of $\fh$ acting on $\fn$ via the adjoint action. Thus,
$R^+=\{ \varepsilon_i-\varepsilon_j \,|\, 1\leq i < j \leq n\}$.
The restriction of the trace form to $\fh$ is non-degenerate. Hence, we have an isomorphism $\fh\mapright{\sim} \fh^*$ given by $h\mapsto (\cdot|h)$. 
This induces a non-degenerate form, also denoted $(\cdot|\cdot)$, on $\fh^*$. 
This form is given by $(\varepsilon_i|\varepsilon_j)=\delta_{ij}$.
% For each $\alpha\in R^+$ the coroot $\alpha^{\vee}\in \fh$ is defined by
% \[ \langle \cdot, \alpha^{\vee} \rangle = \frac{2(\cdot |\alpha)}{(\alpha|\alpha)}. \]
\subsection{}
The weight lattice $P$ is
$P = \{\lambda_1\varepsilon_1 + \cdots + \lambda_n \varepsilon_n \,|\, \lambda_1, \ldots, \lambda_n \in \ZZ\}$.
The cone of dominant weights $P^+$ is
\[P^+ = \{\lambda_1\varepsilon_1 + \cdots + \lambda_n\varepsilon_n \,|\, \lambda_1\geq \lambda_2 \ge \cdots \ge \lambda_n, \, \lambda_1,\ldots, \lambda_n \in \ZZ\}.\]
\subsection{}
The Weyl group $W_0$ (=the symmetric group $S_n$ in the case of $\fgl_n$) acts on $\fh^*$ by permuting $\varepsilon_1,\ldots, \varepsilon_n$. Set $\rho = \sum_{i=1}^{n-1}(n-i)\varepsilon_i$.
The \emph{dot action} of $W_0$ on $\fh^*$ is given by $w\cdot \lambda = w(\lambda+\rho) - \rho$, $w\in W_0$, $\lambda\in\fh^*$.

\subsection{}
The algebra $U'_q=U'_q(\fgl_n)$ is the $\QQ[q,q^{-1}]$ algebra with generators $E_1, \ldots, E_{n-1}$, $K_1^{\pm 1}, \ldots, K_{n-1}^{\pm 1}$, $L_1^{\pm 1},\ldots, L_n^{\pm 1}$, such that $L_1L_2\cdots L_n$ is a central element, and subject to the following relations:
\begin{align*}
L_i E_j &= q^{(\varepsilon_j -\varepsilon_{j+1}|\varepsilon_i)} E_jL_i, \\
L_i F_j &= q^{-(\varepsilon_j-\varepsilon_{j+1}|\varepsilon_i)} F_jL_i, \\
L_iL_i^{-1}&=1=L_i^{-1}L_i, \\
K_iK_i^{-1}&=1=K_i^{-1}K_i, \\
K_i&= L_iL_{i+1}^{-1}, \\
E_iF_j&=F_jE_i \qquad \mbox{if $|i-j|>1$}; \\
E_iF_i &= F_iE_i + \frac{K_i-K_i^{-1}}{q-q^{-1}},
\end{align*}
\[
E_i^2E_{i\pm 1}-(q+q^{-1})E_iE_{i\pm 1}E_i + E_{i\pm 1}E_i^2 = 0,\]
\[F_i^2F_{i\pm 1}-(q+q^{-1})F_iF_{i\pm 1}F_i + F_{i\pm 1}F_i^2 = 0.
\]
\subsection{}
The algebra $U'_q$ is a Hopf-algebra with coproduct $\Delta$, antipode $S$ and counit $\varepsilon$, given by
\[ \Delta(E_i) = E_i\otimes 1 + K_i \otimes E_i, \qquad \Delta(F_i) = F_i\otimes K^{-1}_i + 1 \otimes F_i, \qquad \Delta(L_i)=L_i\otimes L_i; \]
\[ S(E_i) = -K^{-1}_iE_i, \qquad S(F_i)=-F_iK_i, \qquad S(L_i)=L_i^{-1}; \]
\[\varepsilon(E)=\varepsilon(F)=0, \qquad \varepsilon(L_i)=1. \]
\subsection{}
Let $\e$ denote the $U_q$-module defined by
\[ \e=U_q/\ker(\varepsilon\colon U_q\to\CC).\]
Further, for $U'_q$-modules $X,Y$, endow $X\otimes Y$ with the structure of a $U'_q$-module via the coproduct. This endows the category of $U'_q$-modules with a monoidal structure.

\subsection{}Denote by $U_q$ the $\ZZ[q,q^{-1}]$-Hopf subalgebra of $U'_q$ generated by the divided powers
\[ E_i^{(m)} = \frac{E_i^m}{[m]!}, \qquad F_i^{(m)} = \frac{F_i^m}{[m]!}, \qquad m\in\ZZ_{\geq 0}, \]
together with
\[ L_i^{\pm 1} \qquad \mbox{and} \qquad {K_i; c \brack m} = \prod_{j=1}^m \frac{K_iq^{(c-j+1)}-K^{-1}q^{-(c-j+1)}}{q^j-q^{-j}}, \qquad m\in \ZZ_{\geq 0}, c\in\ZZ. \]
Let $U^+_q$ denote the subalgebra of $U_q$ generated by the $E_i^{(m)}$ and denote by $U^-_q$ the subalgebra of $U_q$ generated by the $F_i^{(m)}$. Let $U_q^0$ denote the subalgebra of $U_q$ generated by the $L_i^{\pm 1}$ and the ${K_i; c \brack m}$. Then $U_q$ has the so called triangular decomposition $U_q = U_q^-U_q^0U_q^+$.

\subsection{}Let $V$ be a $U_q$-module, $\gamma$ an anti-endomorphism of $U_q$. Then we may give $\Hom_{\CC}(V,\e)$ the structure of a $U_q$-module via
\[ x\cdot f = \langle f,\gamma(x) - \rangle, \qquad x\in U_q, f\in \Hom_{\CC}(V, \CC). \]
If $\gamma=S$, the antipode, then the resulting module is denote $V^{*}$. Taking $\gamma=S^{-1}$, the resulting module is denoted $V^{\circledast}$.
\subsection{}
Assume $V$ is finite dimensional. Define maps
\[
\varepsilon_V\colon V^*\otimes V \to \e, \quad f\otimes v\mapsto \langle f, v\rangle; \qquad
\eta_V\colon\e\to V\otimes V^*, \quad 1\mapsto \sum_i v_i\otimes v_i^*,
\]
where $\{v_i\}_i$ and $\{v_i^*\}_i$ are dual bases of $V$ and $V^*$ respectively. These maps are $U_q$-module homomorphisms. It is evident that they satisfy \eqref{dualityunit} and \eqref{dualitycounit}.
Similarly, define 
\[
\varepsilon'_V\colon V \otimes V^{\circledast} \to \e, \quad v\otimes f\mapsto \langle f,v\rangle; \qquad
\eta'_V\colon \e\to V^{\circledast}\otimes V, \quad 1\mapsto\sum_i v_i^*\otimes v_i.
\]
These maps are also $U_q$-module homomorphisms and this data satisfies the properties required for right duals.
Thus:
\begin{prop}If $V$ is a finite dimensional module then $V^*$ is a left dual to $V$ and $V^{\circledast}$ is a right dual to $V$.
\end{prop}
\subsection{}Set $K_{\rho} = L_1^{n-1}L_2^{n-2}\cdots L_{n-1}$,
then it follows from the defining relations for $U_q$ that $K^2_{\rho}S(x)=S^{-1}(x)K^2_{\rho}$.
\begin{prop}\label{dualsiso}Let $V$ be a finite dimensional $U_q$-module. The map
\[
\varphi_V\colon V^*\to V^{\circledast}, \qquad
f\mapsto \langle f, K^{-2}_{\rho} - \rangle
\]
is a $U_q$-module isomorphism.
\end{prop}

\begin{proof}Let $f\in V^*$ and let $x\in U_q$, then
\[ \varphi_V(xf) = \langle xf, K^{-2}_{\rho} - \rangle = \langle f, S(x)K^{-2}_{\rho}-\rangle = \langle f, K^{-2}_{\rho}S^{-1}(x)-\rangle = x\varphi_V(f). \]
Thus, $\varphi$ is a $U_q$-module homomorphism. That it is an isomorphism is clear.
\end{proof}

\subsection{}Let $M$ be a $U_q$-module and let $\lambda\in P$. The weight space $M_{\lambda}$ is
\[ M_{\lambda} = \left\{ v\in M\,|\, L_iv = q^{(\lambda|\varepsilon_i)}v,\, {K_i; 0 \brack m}v = {(\lambda|\varepsilon_i) \brack m }v, \mbox{ for all $i$ and $m$}\right\}. \]
The direct sum $\bigoplus_{\lambda\in P} M_{\lambda}$ is a $U_q$-submodule of $M$. It follows from the defining relations that $E_i^{(m)}M_{\lambda} \subseteq M_{\lambda+m(\varepsilon_i-\varepsilon_{i+1})}$ and that $F_i^{(m)}M_{\lambda}\subseteq M_{\lambda-m(\varepsilon_i-\varepsilon_{i+1})}$. A highest weight vector of weight $\lambda$ is a vector $v\in M_{\lambda}$ such that $E_i^{(m)}v = 0$ for all $i$ and $m$. A module $M$ is a highest weight module if it is generated by a highest weight vector.
\subsection{}
Let $\cO$ be the full subcategory of $U_q$-modules $M$ satisfying the following properties:
\begin{itemize}
\item For each $v\in M$, the subspace $U_q^+ v \subset M$ is finite dimensional;
\item For each $\lambda\in P$ the subspace $M_{\lambda}$ is finite dimensional and $M=\bigoplus_{\lambda\in P} M_{\lambda}$;
\item $M$ is finitely generated as a $U_q$-module.
\end{itemize}
Let $M\in\cO$, then $M$ has a finite filtration $0 = M_0 \subset M_1 \subset \cdots \subset M_k = M$
such that each $M_i/M_{i-1}$ is a highest weight module.
\subsection{}
Let $\lambda\in P$ and let $J_{\lambda}$ be the left ideal of $U_q$ generated by
\[ E_i^{(m)}, \qquad L_i-q^{(\lambda|\varepsilon_i)}, \qquad {K_i;0\brack m}-{(\lambda|\varepsilon_i-\varepsilon_{i+1}) \brack m}, \]
for all $i$ and $m$. The Verma module $M(\lambda)$ is
\[ M(\lambda) = U_q/J_{\lambda}. \]
The module $M(\lambda)$ is a highest weight module of weight $\lambda$. In particular, $M(\lambda)\in\cO$. The Verma module $M(\lambda)$ has a unique simple quotient denoted $L(\lambda)$. If $\lambda\in P^+$ then $L(\lambda)$ is finite dimensional. The classes $[M(\lambda)]$, $\lambda\in P$ constitute a basis of $K_0(\cO)$.

\subsection{}Grade the subalgebras $U_q^+$ and $U_q^-$ via elements $\nu=(\nu_1,\ldots, \nu_n)\in \ZZ^n_{\geq 0}$ as follows: the subspace $U_{\nu}^+$ is generated by all the products of $E^{(m_{ij})}_i$s in which for fixed $i$, $\sum_j m_{ij} = \nu_i$. Define $U^-_{\nu}$ similarly.
Let $V, W\in \cO$, define an invertible linear map $\Pi\colon V\otimes W \to V\otimes W$ by $v\otimes w\mapsto q^{-(\lambda|\mu)}w\otimes v$, for $v\in V_{\lambda}$ and $w\in W_{\mu}$.

\begin{prop}[{\cite[Prop.\ 4.1]{Dr}, \cite[Thm.\ 32.1.5]{Lu}, \cite{Ji}}]\label{rmatrixexists}There exists a unique family of elements $P_{\nu}\in U_{\nu}^- \otimes U_{\nu}^+$, $\nu\in\ZZ^n_{\geq 0}$, such that $P_0=1\otimes 1$ and the map
\[
\cR_{VW}\colon V\otimes W \to W\otimes V,\qquad
v\otimes w \mapsto \Pi^{-1}\left( \sum_{\nu\in\ZZ^n_{\geq 0}}P_{\nu}(v\otimes w) \right), \]
is an isomorphism for every $V,W\in\cO$.
\end{prop}

\begin{remarks}$\,$
\begin{enumerate}
\item Our $(\cdot|\cdot)$ is $f$ in \cite{Lu} and our $\cR_{VW}$ is ${}_f\cR_{WV}^{-1}$ in \cite{Lu}.
\item For $V\in\cO$ and $v\in V$, $E_i^{(m)}v=0$ for large enough $m$. Thus, even though the sum $\sum_{\nu\in\ZZ^n_{\geq 0}}P_{\nu}$ is infinite, this is a well defined operator on $V\otimes -$.
\end{enumerate}
 \end{remarks}
\subsection{} 
Let $f\colon V\to V'$ and $g\colon W\to W'$ in $\cO$. Then it is clear that the diagram
\begin{equation*}
					\xymatrixrowsep{1.5pc}\xymatrixcolsep{2pc}\xymatrix{ V\otimes W\ar[d]_{\cR_{VW}}\ar[r]^{f\otimes g}& V'\otimes W'\ar[d]^{\cR_{V'W'}} \\
					W\otimes V\ar[r]_{g\otimes f}&W'\otimes V'}
				\end{equation*}
				commutes. Moreover:
				\begin{prop}\label{RmatrixO}\cite[Prop.\ 32.2.4]{Lu}The following diagrams commute for all $U,V,W\in\cO$
		\begin{equation*}\xymatrixrowsep{1.5pc}\xymatrixcolsep{1pc}\xymatrix{
&U\otimes V\otimes W\ar@/^0pc/[ld]_{\id\otimes \cR_{VW}}\ar@/^0pc/[rd]^{\cR_{(U\otimes V)W}} \\
U\otimes W\otimes V \ar[rr]_{\cR_{UW}\otimes\id}&&W\otimes U\otimes V
}\end{equation*}
\begin{equation*}\xymatrixrowsep{1.5pc}\xymatrixcolsep{1pc}\xymatrix{
&U\otimes V\otimes W\ar@/0pc/[ld]_{\cR_{UV}\otimes\id}\ar@/0pc/[rd]^{\cR_{U(V\otimes W)}}\\
V\otimes U\otimes W \ar[rr]_{\id\otimes\cR_{UW}}&&V\otimes W\otimes U. 
}\end{equation*}
\end{prop}
\subsection{}
Define a linear map
\[
\varphi\colon U_q \otimes U_q \to U_q, \qquad
x\otimes y\mapsto S(y)x,
\]
where $S$ is the antipode.
Let $M\in\cO$. Let $P_{\nu}, \nu\in\ZZ^{n}_{\geq 0}$ be as in Prop.\ \ref{rmatrixexists}. Define
\[
\theta_M\colon M \to M, \qquad
v\mapsto q^{(\lambda|\lambda+2\rho)}\sum_{\nu\in\ZZ^n_{\geq 0}} \varphi(P_{\nu}) v.
\]
This is a well defined linear operator and by \cite[Prop.\ 6.1.7]{Lu}, $\theta_M$ is a $U_q$-module automorphism. Now suppose $M$ is generated by a highest weight vector $v_{\lambda}^+\in M_{\lambda}$. Then by definition of the $P_{\nu}$,
\[ \theta_M(v_{\lambda}^+) = q^{(\lambda|\lambda+2\rho)}v_{\lambda}^+ + \mbox{terms of lower weight}. \]
As $\End_{U_q}(M)=\CC$, it follows that $\theta_M$ is multiplication by $q^{(\lambda|\lambda+2\rho)}$.

\begin{prop}\label{ribbonO}Let $V,W\in\cO$. Then
$\theta_{V\otimes W} \circ (\theta_V\otimes \theta_W)^{-1} =\cR_{WV}\circ \cR_{VW}$.
\end{prop}

\begin{proof}The operator $\Pi^{-1} \circ \sum_{\nu\in\ZZ^n_{\geq 0}} P_{\nu}$ may be interpreted as an element of an appropriate completion of $U_q\otimes U_q$ (see \cite[\S4.1]{Lu}). Now the proof of our assertion is exactly the same as that of \cite[Prop.\ 3.2]{Dr}.
\end{proof}

\section{$\fsl_2$-categorifications and derived equivalences for generic $q$}\label{s:generic}
In this section we assume $q$ is not a root of $1$.
Let $V=L(\varepsilon_1)$. Then $V$ has weights $\varepsilon_1, \ldots, \varepsilon_n$ and $V^*$ has weights $-\varepsilon_1, \ldots, -\varepsilon_n$. It is clear that $V, V^*\in \cO$.

\begin{prop}\label{tensoridentity}Let $\lambda\in\fh^*$. Then $V\otimes M(\lambda)$ has a filtration with quotients isomorphic to $M(\lambda+\varepsilon_i)$, $i=1,\ldots, n$. Similarly, $V^*\otimes M(\lambda)$ has a filtration with quotients isomorphic to $M(\lambda-\varepsilon_i)$, $i=1,\ldots, n$.
\end{prop}\begin{proof}
This is a special case of \cite[Prop.\ 2.16]{APW}.
\end{proof}

\begin{cor}The category $\cO$ is stable under the functors $V\otimes -$ and $V^*\otimes -$.
\end{cor}
\subsection{}
Let $M\in\cO$. Let
\[ Y_M(V\otimes M) = \cR_{MV}\circ \cR_{VM}(V\otimes M). \]
Define $\pr_i(V\otimes M)$ to be the generalized eigenspace of of $Y_M$ with eigenvalue $q^i$. That is,
\[ \pr_i(V\otimes M) = \varinjlim_m \ker((Y_M-q^i)^m\colon V\otimes M \to V\otimes M). \]
By definition,
$Y_M = \cR_{MV}\cR_{VM} = \theta_{V\otimes M}\circ (\theta_V \otimes \theta_M)^{-1}$.
As each object of $\cO$ has a finite filtration by highest weight modules and $\theta$ acts by scalar multiplication on highest weight modules, the direct limit in the definition above must stabilize after finitely many terms. 
\subsection{}
For each $a\in\ZZ$, define a functor
\[
E_a\colon \cO \to \cO, \qquad
M \mapsto \pr_{2a}(V\otimes M).
\]
As $\theta_V$ is multiplication by $q^{(\varepsilon_1|\varepsilon_1+2\rho)}=q^{2n-1}$, we infer that
\[ E_a = \bigoplus_{j\in \ZZ} \pr_{2a+j+2n-1}\circ (V \otimes -) \circ \pr_j. \]
Furthermore, if $\lambda = \lambda_1\varepsilon_1 + \cdots + \lambda_n\varepsilon_n$, then
\begin{equation}\label{contentcalculation} (\lambda+\varepsilon_i|\lambda+\varepsilon_i+2\rho) - (\lambda|\lambda+2\rho)-(\varepsilon_1|\varepsilon_1+2\rho) = 2(\lambda_i-i+2). \end{equation}
Since $\theta_{M(\lambda)} = q^{(\lambda|\lambda+2\rho)}$, it follows that
$V\otimes M = \bigoplus_{a\in\ZZ} E_a(M)$ for all $M\in \cO$.
\subsection{}
Define
\[ F_a = \bigoplus_{j\in\ZZ} \pr_j \circ (V^*\otimes -)\circ \pr_{2a+j+2n-1}.\]
Since $\pr_j$ is left and right adjoint to itself, it follows that $F_a$ is left and right adjoint to $E_a$ (see Prop. \ref{dualsiso}).
\subsection{}
Let $\lambda=\lambda_1\varepsilon_1 + \cdots +\lambda_n\varepsilon_n$ and $\mu=\mu_1\varepsilon_1 + \cdots + \varepsilon_n$ be in $P$. Write $\lambda\to_a \mu$ if there exists $j$ such that $\lambda_j-j+1=a-1$, $\mu_j-j+1=a$ and $\lambda_i=\mu_i$ for $i\neq j$. Then \eqref{contentcalculation} along with Prop.\ \ref{tensoridentity} implies that
\[ [E_aM(\lambda)] = \sum_{\lambda\to_a\mu} [M(\mu)], \qquad [F_aM(\lambda)] = \sum_{\mu\to_a\lambda} [M(\mu)] \]
in $K_0(\cO)$. Hence
\[ [E_aF_aM(\lambda)]-[F_aE_aM(\lambda)] = c_{\lambda,a}[M(\lambda)], \]
where
\[ c_{\lambda,a} = \#\{i\,|\, \lambda_i-i+1 =a\} - \#\{i\,|\, \lambda_i - i=a\}.\]
\begin{prop}The functors $E_a$, $F_a$ define a weak $\fsl_2$-categorification on $\cO$.
\end{prop}

\begin{proof}As the classes of Verma modules are a basis for $K_0(\cO)$, in view of the above discussion, all that remains to be checked is that the classes of simple modules are weight vectors in $K_0(\cO)$. Let $[L(\lambda)] = \sum_{\mu\in P}a_{\lambda,\mu}[M(\mu)]$.
Then we infer from the Linkage Principle \cite[Thm.\ 9.1.8]{CP} that if $a_{\lambda,\mu}\neq 0$, then $\mu\in W_0\cdot\lambda$. By reducing to the case of simple reflections, we deduce that $c_{\lambda,a} = c_{w\cdot\lambda,a}$.
Thus, $[L(\lambda)]$ has weight $c_{\lambda,a}$.
\end{proof}

\begin{remark}The Linkage Principle, as stated in \cite[Thm.\ 9.1.8]{CP}, applies only to the quantum group associated to a simple Lie algebra. However, it remains valid for $U_q(\fgl_n)$. One way of seeing this is by using the embedding $U_q(\fsl_n)\hookrightarrow U_q(\fgl_n)$ (see \cite[\S12.C]{CP}).
\end{remark}

\begin{prop}\label{heckerelation}The morphism
$(\cR_{VV}-q)\circ(\cR_{VV}+q^{-1})\colon V\otimes V \to V\otimes V$
is zero.
\end{prop}

\begin{proof}
The module $V\otimes V$ decomposes as
$V\otimes V = L(2\varepsilon)\oplus L(\varepsilon_1+\varepsilon_2)$.
Since $\cR_{VV}^2=\theta_{V\otimes V}(\theta_V\otimes \theta_V)^{-1}$,
\[ (2\varepsilon_1|2\varepsilon_1+2\rho)-2(\varepsilon_1|\varepsilon_1+2\rho)=2 \qquad\mbox{and} \qquad (\varepsilon_1+\varepsilon_2|2\varepsilon_1+2\rho)=-2, \]
we infer that $\cR_{VV}^2 - q^2$ is zero on the $L(2\varepsilon)$ component of $V\otimes V$ and $\cR_{VV}^2-q^{-2}$ is zero on the $L(\varepsilon_1 + \varepsilon_2)$ component of $V\otimes V$. Using the specialization $q\to 1$ and arguing as in \cite[Prop.\ 4.4]{LR}, we deduce that $\cR_{VV}$ is multiplication by $q$ on the $L(2\varepsilon_1)$ component, and is multiplication by $-q^{-1}$ on the $L(\varepsilon_1+\varepsilon_2)$ component of $V\otimes V$.
\end{proof}

\begin{prop}
Let $F_V, \sigma_i, Y_i$ be as in \S\ref{braidrels} with $V=L(\varepsilon_1)$.
Define
\begin{align*}
\Phi\colon H_{k+1} &\to \End(F_V^{k+1}), \\
T_i &\mapsto q\sigma_i, \\
X_i &\mapsto Y_i, \\
q_0 &\mapsto q^2.
\end{align*}Then $\Phi$ defines a representation of $H_{k+1}$ on $F_V^{k+1}$.
\end{prop}

\begin{proof}Immediate from Prop.\ \ref{braidrels}, Prop.\ \ref{RmatrixO}, Prop.\ \ref{ribbonO} and Prop.\ \ref{heckerelation}.
\end{proof}
\subsection{}
Let $F_V$ be as in \S\ref{braidrels} with $V=L(\varepsilon_1)$. Then there is an action of $H_2$ on $F_V^2$. Let $M\in\cO$, it follows from the definitions that
\[ E^2_a(M) = \{ v\in V\otimes V \otimes M\,|\, (X_2-a)^N(X_1-a)^Nv=0, \mbox{ for sufficiently large $N$}\}.\]
Combining this with Prop.\ \ref{commutation1} we deduce that the action of $H_2$ on $F_V^2$ restricts to an action of $H_2$ on $E^2$. Hence:
\begin{prop}\label{categorificationgenericq}The functors $E=E_a$ and $F=F_a$ along with the morphisms $T=T_1$ and $X=X_1$ are an $\fsl_2$-categorification.
\end{prop}
\subsection{}
Let $E=E_a$ and $F=F_a$. For $m\in\ZZ$, let $\cO_m$ denote the subcategory of $\cO$ consisting of those modules $V$ such that
$[EF(V)]-[FE(V)]=m[V]$
in $K_0(\cO)$. Let $\Theta$ be as in \S\ref{ss:rickardcomplex}. Then Thm.\ \ref{reflectionequivalence} gives:
\begin{cor}The complex of functors $\Theta$ induces a self-equivalence of $\Hob(\cO)$ and of $\Db(\cO)$. By restriction $\Theta$ induces equivalences $\Hob(\cO_{-m})\mapright{\sim}\Hob(\cO_m)$ and $\Db(\cO_{-m})\mapright{\sim}\Db(\cO_m)$. Furthermore, the operator on $K_0(\cO)$ induced by $\Theta$ coincides with the reflection $s$ on $K_0(\cO)$ (viewed as an $\fsl_2$-module).
\end{cor}
\subsection{}
Let $M\in\cO$ be indecomposable. By the Linkage Principle \cite[Thm.\ 9.1.8]{CP}, if $L(\lambda)$ and $L(\mu)$ are composition factors of $M$, then $\lambda$ and $\mu$ are in the same dot-orbit of $W_0$. For $\lambda\in P$, let $\cO_{\lambda}$ be the Serre subcategory of $\cO$ generated by simples of the form $L(w\cdot \lambda)$, $w\in W_0$.
The following result is proved along the same lines as \cite[Thm.\ 7.24]{ChRo}.
\begin{thm}\label{dequivgenericq}Let $\lambda,\mu\in P$ have the same stabilizer under the dot action of $W_0$. Then there are equivalences
\[ \Hob(\cO_{\lambda})\mapright{\sim} \Hob(\cO_{\mu}) \qquad\mbox{and}\qquad
\Db(\cO_{\lambda})\mapright{\sim} \Db(\cO_{\mu}) \]
that induce the map
\[ [M(w\cdot \lambda)] \mapsto [M(w\cdot \mu)], \qquad w\in W_0, \]
at the level of Grothendieck groups.
\end{thm}

\begin{proof}
For $i\in\ZZ$, define $t_i\colon \ZZ \to \ZZ$ by 
\[ t_i(m) = \begin{cases} m+1 & \mbox{if $m=i-1$}, \\
m-1  & \mbox{if $m=i$}, \\
m & \mbox{otherwise},\end{cases}\]
and define $d\colon\ZZ\to \ZZ$ by $d(m) = m+1$. 
Let $T$ be the group generated by $d$ and $t_i, i\in\ZZ$.
Identify $P$ with $\ZZ^n$ via
$\lambda_1\varepsilon_1 + \cdots + \lambda_n\varepsilon_n \mapsto (\lambda_1,\ldots,\lambda_n)$.
This defines an action of $T$ on $P$ via the diagonal action of $T$ on $\ZZ^n$. It is clear that the regular action of $W_0$ on $P$ commutes with the action of $T$.
We claim that two elements $\nu,\nu'\in P$ have the same stabilizer in $W_0$ (regular action) if and only if they are in the same $T$-orbit. Indeed, both conditions are equivalent to the following: if $\nu=\nu_1\varepsilon_1 + \cdots +\nu_n\varepsilon_n$ and $\nu'=\nu'_1\varepsilon_1+\cdots+\nu'_n\varepsilon_n$, then for all $i$ and $j$, $\nu_i-\nu_j=0$ if and only if $\nu'_i-\nu'_j=0$.

Thus, if $\lambda$ and $\mu$ have the same $W_0$ stabilizer under the dot action, then $t(\lambda+\rho) = \mu+\rho$, for some some $t\in T$. Without loss of generality, we may assume that $t=d$ or $t=t_a$ for some $a\in\ZZ$.
If $t=d$, then consider the module $V'=L(\varepsilon_1 + \cdots + \varepsilon_n)$. The module $V'$ is one dimensional and it follows that the functor $V'\otimes - \colon\cO\to\cO$ is an equivalence. Moreover, the induced map at the level of $K_0(\cO)$ is precisely $[M(w\cdot \lambda)]\mapsto [M(w\cdot \mu)]$.
If $t=t_a$, then let $\Theta_a$ be the complex of functors obtained from the $\fsl_2$ categorification $E=E_a$ and $F=F_a$. Let $s$ denote the reflection that $\Theta_a$ induces on $K_0(\cO)$ (viewed as an $\fsl_2$-module). Define an $\fsl_2$-module $U_a=\bigoplus_{i\in\ZZ} u_i$ as follows:
\[ eu_{a-1}=u_a, \quad fu_a = u_{a-1}, \quad eu_i=0 \quad\mbox{if $i\neq a-1$}, \quad fu_i=0 \quad\mbox{if $i\neq a$}.\]
That is, $U_a$ is the direct sum of the defining representation of $\fsl_2$ and infinitely many copies of the trivial module. Thus, on $U_a$ the reflection $s$ is given by
\[ su_{a-1}=u_a,\quad su_a = -u_{a-1}, \quad su_i=u_i\quad\mbox{if $i\neq a,a-1$}.\]
Thus, on the tensor product $U^{\otimes n}_a$, we have
$su_{\nu} = (-1)^{h^-(\nu)}u_{t_a\nu}$,
where
$u_{\nu} = u_{\nu_1}\otimes \cdots \otimes u_{\nu_n}$ and
$h^-(\nu)=\#\{i\,|\, \nu_i=a\}$.
The map
\[ U_a^{\otimes n} \to K_0(\cO), \qquad u_{\nu+\rho}\mapsto [M(\nu)], \]
is an $\fsl_2$-module homomorphism. Hence
\[ s[M(\nu)] = (-1)^{h^-(\nu+\rho)} [M(t_a(\nu+\rho)-\rho)].\]
Note that
\[ h^-(\lambda+\rho) = h^-(w(\lambda+\rho)) = h^-(w\cdot \lambda + \rho), \]
for all $w\in W_0$. Thus, if we let $h^-=h^-(\lambda+\rho)$, then $\Theta_a[h^-]$ restricts to an equivalence $\Db(\cO_{\lambda})\mapright{\sim} \Db(\cO_{\mu})$ that induces the map $[M(w\cdot\lambda)] \mapsto [M(w\cdot \mu)]$, $w\in W_0$,
at the level of Grothendieck groups.
\end{proof}

\begin{remark}Combining the various $\fsl_2$-categorifications $(E_a, F_a)$, $a\in\ZZ$, gives a categorification of the affine Lie algebra $\widehat{\fsl}_{\infty}$ in the sense of \cite{Ro}, see \cite[\S5]{W} and \cite{BK}. Also cf.\ \cite{LLT}.
\end{remark}

\section{$\fsl_2$-categorifications and derived equivalences for $q$ a root of $1$}
In this section $q$ will be an $\ell$th root of unity, $\ell\neq 2$. We will allow $\ell=\infty$, this corresponds to the case of generic $q$. However, since the categories we will now construct $\fsl_2$-categorifications on are different from $\cO$, the results of this section do not subsume those of the previous one. 

\subsection{}Let $\cC$ be the full subcategory of $\cO$ consisting of finite dimensional modules. Let $V\in\cC$. Then it follows from \cite[Lemma 1.13]{APW} that if $V_{\lambda}\neq 0$ then $V_{w(\lambda)}\neq 0$, for all $w\in W_0$. The following result is classical.
\begin{lemma}Let $\lambda\in P$ and let $M(\lambda)$ be the corresponding Verma module. Let $\Sigma$ be the set of $U_q$-submodules $K$ of $M(\lambda)$ such that $M(\lambda)/K$ is finite dimensional. Then $\Sigma$ has a unique minimal element.
\end{lemma}

\begin{proof}
All weights of $M(\lambda)$ belong to $\pi=\lambda-\ZZ_{\geq 0}R^+$. Let $\pi' = W_0(P^+\cap \pi)$ and let $N(\lambda)$ denote the $U_q$-module generated by $\bigoplus_{\mu\not\in\pi'} M(\lambda)_{\mu}$. If $K\in\Sigma$, then $N(\lambda)\subseteq K$. Furthermore, the set $\pi'$ is finite and the weight spaces of $M(\lambda)$ are finite dimensional, thus $M(\lambda)/N(\lambda)\in \Sigma$. Hence, $N(\lambda)$ is the required minimal element.
\end{proof}
\subsection{}
Let $N(\lambda)\subseteq M(\lambda)$ be the minimal element of the lemma above. The \emph{Weyl module} $\Delta(\lambda)$ is defined by
\[ \Delta(\lambda) = M(\lambda)/N(\lambda). \]
By construction, $\Delta(\lambda)=0$ unless $\lambda\in P^+$. In fact, $\Delta(\lambda)\neq 0$ if and only if $\lambda\in P^+$.
The module $\Delta(\lambda)$ is a highest weight module. Further, any other highest weight module with highest weight $\lambda$ is necessarily isomorphic to a quotient of $\Delta(\lambda)$. The module $\Delta(\lambda)$ has a unique simple quotient, denoted $L_q(\lambda)$. Every simple module in $\cC$ is isomorphic to $L_q(\lambda)$ for some $\lambda\in P^+$.

\subsection{}
Let $s_i\in W_0$ denote the transposition $(\varepsilon_i, \varepsilon_{i+1})$. The \emph{affine Weyl group} is generated by the affine reflections $s_{i,m}$, $i=1,\ldots, n-1$, $m\in \ZZ$, where $s_{i,m}(\lambda)=s_i(\lambda) + m\ell(\varepsilon_i - \varepsilon_{i+1})$, $\lambda\in \fh^*$. The dot action of $W_{\ell}$ on $\fh^*$ is defined as $w\cdot\lambda = w(\lambda+\rho) - \rho$, $w\in W_{\ell}$, $\lambda\in\fh^*$. The following result is the Linkage Principle for $q$ a root of unity.
\begin{prop}\cite[Thm.\ 8.1]{APW}Let $\lambda,\mu\in P^+$. If $L(\mu)$ occurs as a composition factor of $\Delta(\lambda)$, then $\mu\in W_{\ell}\cdot \lambda$.
\end{prop}

\subsection{}
Let $V=\Delta(\varepsilon_1)$. Then $V$ has weights $\varepsilon_1, \ldots, \varepsilon_n$ and $V^*$ has weights $-\varepsilon_1, \ldots, -\varepsilon_n$. It is clear that $\cC$ is stable under tensoring with $V$ and $V^*$.

\begin{prop}Let $\lambda\in P^+$. Then $V\otimes \Delta(\lambda)$ has a filtrations with quotients isomorphic to $\Delta(\lambda+\varepsilon_i)$, $i=1,\ldots, n$. Similarly, $V^*\otimes \Delta(\lambda)$ has a filtration with quotients isomorphic to $\Delta(\lambda-\varepsilon_i)$, $i=1,\ldots, n$.
\end{prop}

\begin{proof}This is a special case of \cite[Prop.\ 2.16]{APW}.
\end{proof}

\begin{remark}It is understood in the Proposition above that $\Delta(\mu)=0$ if $\mu\not\in P^+$.
\end{remark}
\subsection{}
Let $M\in\cC$. Let
\[ Y_M(V\otimes M) = \cR_{MV}\otimes \cR_{VM}(V\otimes M). \]
Let $\pr_i(V\otimes M)$ be the generalized eigenspace of $Y_M$ with eigenvalue $q^i$. For each $a\in\ZZ/\ell\ZZ$, define
\[
E_a\colon \cC \to \cC, \qquad
M\mapsto \pr_{2a}(V\otimes M).
\]
As in \S\ref{s:generic}, 
$E_a = \bigoplus_{j\in\ZZ/\ell\ZZ} \pr_{2a+j+2n-1}\circ (V\otimes - )\circ \pr_j$,
and the functor
\[ F_a = \bigoplus_{j\in\ZZ/\ell\ZZ} \pr_j \circ (V^*\otimes -)\circ\pr_{2a+j+2n-1} \]
is left and right adjoint to $E_a$.
\subsection{}
Let $\lambda=\lambda_1\varepsilon_1 + \cdots +\varepsilon_n$ and $\mu=\mu_1\varepsilon_1 + \cdots + \mu_n\varepsilon_n$ be in $P^+$. Write $\lambda\to_a \mu$ if there exists $j$ such that $\lambda_j-j+1\equiv a-1 \mod \ell$, $\mu_j-j+1\equiv a\mod \ell$ and $\lambda_i \equiv \mu_i \mod \ell$ for $i\neq j$. Then
\[ [E_a\Delta(\lambda)] = \sum_{\lambda \to_a \mu} [\Delta(\mu)], \qquad [F_a\Delta(\lambda)] = \sum_{\mu\to_a\lambda}[\Delta(\mu)], \]
in $K_0(\cC)$. Hence
\[ [E_aF_a\Delta(\lambda)] - [F_aE_a\Delta(\lambda)] = c_{\lambda,a}[\Delta(\lambda)], \]
where
\begin{align*}
c_{\lambda, a} &= \#\{ i\,|\, \lambda_i-i+1 \equiv a\mod\ell\mbox{ and }\lambda+\varepsilon_i \in P^+\} \\
&\quad - \#\{ i \,|\, \lambda_i -i \equiv a\mod\ell\mbox{ and }\lambda-\varepsilon_i \in P^+\}. \end{align*}
As in \S\ref{s:generic}, we have
\begin{prop}\label{categorificationrootsof1}The functors $E_a$ and $F_a$ define an $\fsl_2$-categorification.
\end{prop}
Let $\Theta$ be as in \S\ref{ss:rickardcomplex}, then we also have:
\begin{cor}The complex of functors $\Theta$ induces a self-equivalence of $\Hob(\cC)$ and of $\Db(\cC)$. By restriction $\Theta$ induces equivalences $\Hob(\cC_{-m})\mapright{\sim} \Hob(\cC_m)$ and $\Db(\cC_{-m})\mapright{\sim} \Db(\cC_m)$. Furthermore, the operator on $K_0(\cC)$ induced by $\Theta$ coincides with the reflection $s$ on $K_0(\cC)$ (viewed as an $\fsl_2$-module).
\end{cor}

\begin{remark}Combining the various $\fsl_2$-categorifications $(E_a, F_a)$, $a\in\ZZ/\ell\ZZ$, gives a categorification of the affine Lie algebra $\widehat{\fsl}_{\ell}$ in the sense of \cite{Ro}, see \cite[\S5]{W} and \cite{BK}. Also cf.\ \cite{LLT}.
\end{remark}

\end{document}